\documentclass[11pt]{amsart}
\usepackage{amscd,amssymb,stmaryrd,url}
\usepackage[all]{xy}

\title{Quasi-hereditary algebras and generalized Koszul duality}
\author{Dag Oskar Madsen}

\address{Faculty of Professional Studies, University of Nordland, NO-8049 Bod{\o}, Norway}
\email{Dag.Oskar.Madsen@uin.no}

\newtheorem{lem}{Lemma}[section]
\newtheorem{prop}[lem]{Proposition}
\newtheorem{cor}[lem]{Corollary}
\newtheorem{thm}[lem]{Theorem}

\newtheorem*{cor2}{Corollary}
\theoremstyle{definition}

\newtheorem{example}[lem]{Example}

\newtheorem*{defin2}{Definition}

\newcommand{\G}{\Gamma}

\renewcommand{\L}{\Lambda}
\newcommand{\Z}{\mathbb Z}

\newcommand{\End}{\operatorname{End}}
\newcommand{\Ext}{\operatorname{Ext}}
\newcommand{\gldim}{\operatorname{gldim}}
\newcommand{\Gr}{\operatorname{\mathsf{Gr}}}
\newcommand{\gr}{\operatorname{\mathsf{gr}}}
\newcommand{\Hom}{\operatorname{Hom}}

\renewcommand{\mod}{\operatorname{\mathsf{mod}}}

\newcommand{\op}{\operatorname{op}}

\renewcommand{\top}{\operatorname{top}}

\newcommand{\gsh}[1]{\langle #1 \rangle}

\begin{document}

\begin{abstract}
We present an easily applicable sufficient condition for standard Koszul algebras to be Koszul with respect to $\Delta$. If a quasi-hereditary algebra $\L$ is Koszul with respect to $\Delta$, then $\L$ and the Yoneda extension algebra of $\Delta$ are Koszul dual in a sense explained below, implying in particular that their bounded derived categories of finitely generated graded modules are equivalent. We also prove that the extension algebra of $\Delta$ is Koszul in the classical sense.
\end{abstract}

\maketitle

\section*{Introduction}

Finite dimensional algebras appearing in the representation theory of algebraic groups and Lie algebras usually have nice homological properties. For instance, such algebras are often Koszul and quasi-hereditary \cite{Icra}. The work presented here is inspired by the paper \cite{Stan} and what the authors of that paper refer to as ``Ovsienko's idea'': For some quasi-hereditary Koszul algebras $\L$ it might be possible to construct an equivalence of bounded derived categories of finitely generated graded modules $\mathcal D^b(\gr \L) \to \mathcal D^b(\gr \G)$, where $\G$ is the extension algebra of standard modules $\G=[\Ext_{\L}^{\ast} (\Delta,\Delta)]^{\op}$. The equivalence should send standard modules to projective modules and costandard modules to simple modules. The extension algebra of standard modules is always directed \cite[Theorem 1.8(b)]{Icra}, which often means its representation theory is easier to understand than that of the original algebra $\L$.

In \cite{Stan} Ovsienko's idea was realized for a large class of quasi-hereditary Koszul algebras, but at the price of having to organize extensions of $\Delta$ as a Koszul category, not as an algebra. In the present paper we use a different approach and pay a different price. When Ovsienko's idea works for an algebra, we interpret this as an instance of generalized Koszul duality \cite{TKos} \cite{Adv}. The drawback is that we have to regrade our algebra $\L$ to make everything fit. This regrading might at first seem a bit artificial, but generalized Koszul duality tells us that $\L$ is isomorphic to an extension algebra over $\Gamma$ and the new grading agrees with the $\Ext$-grading. The advantages of our approach are this symmetry between $\L$ and $\G$ and the connection with the larger framework of generalized Koszul duality.

For which quasi-hereditary algebras $\L$ do we in this way obtain an equivalence $\mathcal D^b(\gr \L) \to \mathcal D^b(\gr \G)$? First of all $\L$ should not only be Koszul, but standard Koszul \cite{QExt}, meaning that the standard modules have linear projective resolutions. If in addition the radical layers of the standard modules are well behaved (a certain condition (H) is satisfied), then we get our main result.

\begin{cor2}[to Theorem \ref{hoved}]
Let $\L$ be a standard Koszul algebra admitting a height function $h$ satisfying condition {\rm{(H)}}. Let $$\G=[\Ext_{\L}^{\ast} (\Delta,\Delta)]^{\op}.$$ Then $$\L \cong [\Ext_{\G}^{\ast} (D\Delta,D\Delta)]^{\op}$$ as ungraded algebras. Furthermore, when $\L$ and $\G$ are given the $\Ext$-grading, then there is an equivalence of triangulated categories $\mathcal D^b(\gr \L) \to \mathcal D^b(\gr \G)$.
\end{cor2}

We also prove in Theorem \ref{ord} that $\G=[\Ext_{\L}^{\ast} (\Delta,\Delta)]^{\op}$ is Koszul in the classical sense.

In section 1 we recall the basic notions and facts about quasi-hereditary algebras. In section 2 we look at the connection between standard Koszul algebras and algebras Koszul with respect to $\Delta$. In section 3 we introduce condition (H) for height functions and discuss its consequences. In section 4 we present our main result and apply it to a wide range of examples.

We wish to thank the referee for valuable remarks and for pointing us to relevant additional references.

\section{Quasi-hereditary algebras with duality}

Let $\Bbbk$ be a field and let $\L$ be a finite dimensional $\Bbbk$-algebra. Fix an ordering on a complete set of non-isomorphic simple $\L$-modules $S_1, \ldots, S_r$. For each $0 \leq i \leq r$, define the \emph{standard module} $\Delta_i$ to be the largest quotient of the projective module $P_i$ having no simple composition factors $S_j$ with $j>i$. Dually, define the \emph{costandard module} $\nabla_i$ to be the largest submodule of the injective module $I_i$ having no simple composition factors $S_j$ with $j>i$. Let $$\Delta=\bigoplus_{i=1}^r \Delta_i$$ and $\nabla=\bigoplus_{i=1}^r \nabla_i$.

We say that $\L$ is a \emph{quasi-hereditary algebra} \cite{Hwc} \cite{QuaH} if (i) $\Lambda$ admits a $\Delta$-filtration, i.e., there is a filtration $0=M_0 \subseteq M_1 \subseteq \ldots \subseteq M_t=\L$ where the subfactors $M_j/M_{j-1}$ are standard modules for all $1 \leq j \leq t$, and (ii) $\End_\L (\Delta_i)$ is a division ring for all $1 \leq i \leq r$. All quasi-hereditary algebras have finite global dimension.

If $\L$ is quasi-hereditary, then for all $1 \leq i,j \leq r$ and $n>0$ we have $\Ext^n_\L(\Delta_i,\nabla_j)=0$ \cite[Theorem 1.8]{Icra}. Furthermore $\Hom_\L(\Delta_i,\nabla_j)=0$ whenever $i \neq j$, and $\Hom_\L(\Delta_i,\nabla_i) \simeq \End(S_i)$.

A quasi-hereditary structure on an algebra determines a unique \emph{characteristic tilting module} $\mathbb T$, a basic tilting-cotilting module that admits both a $\Delta$-filtration and a $\nabla$-filtration \cite{Tilt}. For each $1 \leq i \leq r$, the module $\mathbb T$ has an indecomposable direct summand $T_i$ which admits morphisms $\Delta_i \hookrightarrow T_i \twoheadrightarrow \nabla_i$ with non-zero composition. These are the only indecomposable direct summands of $\mathbb T$, so $\mathbb T \simeq \bigoplus_{i=1}^r T_i$.

In this paper we only consider \emph{quasi-hereditary algebras with duality}, that is, we suppose our quasi-hereditary algebras are equipped with a contravariant exact equivalence $(-)^\circ \colon \mod \L \to \mod \L$ such that ${S_i}^\circ \simeq S_i$ for all $1 \leq i \leq r$. The direct summands of the characteristic tilting module have the property that ${T_i}^\circ \simeq T_i$ for all $1 \leq i \leq r$. Quasi-hereditary algebras with duality are called \emph{BGG algebras} in \cite{Bgg}. For such algebras the BGG reciprocity principle holds \cite{Bgg}; we have $$(P_i \colon \Delta_j)=[\Delta_j \colon S_i]$$ for all $1 \leq i ,j \leq r$, where $(P_i \colon \Delta_j)$ denotes the filtration multiplicity of $\Delta_j$ in a $\Delta$-filtration of $P_i$ and $[\Delta_j \colon S_i]$ denotes the multiplicity of $S_i$ in a composition series for $\Delta_j$.

\section{Standard Koszul algebras and $T$-Koszul algebras}

We are interested in algebras that are simultaneously quasi-hereditary and  Koszul. For this we need a graded setting, so let $\L=\bigoplus_{i=0}^t \L_i$ be a graded finite dimensional algebra with $\L_0 \cong \Bbbk^{\times r}$. In this case the Jacobson radical of the algebra is given by $J=\bigoplus_{i=1}^t \L_i$. We keep the notation from \cite{Adv}, in particular the category of graded $\L$-modules is denoted by $\Gr \L$, and the category of finitely generated graded $\L$-modules is denoted by $\gr \L$. Given a graded $\L$-module $M$, the \emph{$j$th graded shift} of M, denoted $M \gsh j$, is the module with graded parts $(M \gsh j)_i= M_{i-j}$ and module structure inherited from $M$. A graded algebra $\L=\bigoplus_{i=0}^t \L_i$ is called a \emph{(classical) Koszul algebra} \cite{Pri} \cite{Bei} if $\Ext^i_{\Gr \L}(\L_0,\L_0 \gsh j)=0$ whenever $i \neq j$.

Suppose $\L$ is quasi-hereditary. For all $1 \leq j \leq r$, the modules $\Delta_j$, $\nabla_j$ and $T_j$ have graded lifts which preserve the morphisms $\Delta_j \hookrightarrow T_j \twoheadrightarrow \nabla_j$, see for example \cite{Stan}. The functor $(-)^{\circ}$ lifts to a duality $(-)^{\circ} \colon \gr (\L) \to \gr (\L)$ with $S_j \gsh l^\circ \simeq S_j \gsh {-l}$ for all $l \in \Z$. A quasi-hereditary algebra with duality is called \emph{standard Koszul} \cite{QExt} if standard modules have linear projective resolutions, in other words $\Ext^i_{\Gr \L}(\Delta,\L_0 \gsh j)=0$ whenever $i \neq j$. Standard Koszul algebras are Koszul in the ordinary sense \cite[Theorem 1,4]{QExt}.

In the present paper we explore a connection between standard Koszul algebras and the $T$-Koszul algebras which were first defined in \cite{TKos}. In \cite{Adv} we proposed the following simplified definition. The context is graded algebras as above, except that $\L_0$ is no longer required to be semi-simple. We assume $\dim_k\L_i < \infty$ for all $i \geq 0$.

\begin{defin2}
Let $\L=\bigoplus_{i\geq 0} \L_i$ be a graded algebra with $\gldim \L_0 < \infty$, and let $T$ be a graded $\L$-module concentrated in degree zero.
We say that $\L$ is \emph{Koszul with respect to $T$} or \emph{$T$-Koszul} if both of the following conditions hold.
\begin{itemize}
\item[(i)] $T$ is a tilting $\L_0$-module.
\item[(ii)] $T$ is graded self-orthogonal as a $\L$-module, that is, we have $$\Ext^i_{\Gr \L}(T,T \gsh j)=0 \text{ whenever }i \neq j.$$
\end{itemize}
\end{defin2}

The following theorem is a statement of Koszul duality for $T$-Koszul algebras. Here $D$ denotes the functor $D=\Hom_\Bbbk(-,\Bbbk)$.

\begin{thm}\cite[Theorem 4.2.1]{Adv}\label{kosdua}
Let $\L=\bigoplus_{i\geq 0} \L_i$ be a graded algebra with $\gldim \L_0 < \infty$. Suppose $\L$ is a Koszul algebra with respect to a module $T$. Let $\G=[\Ext_\L^{\ast} (T,T)]^{\op}$. Then
\begin{itemize}
\item[(a)]$\gldim \G_0 < \infty$, and $\G$ is a Koszul algebra with respect to $_{\G}DT$.
\item[(b)] There is an isomorphism of graded algebras $\L \cong [\Ext_\G^{\ast} (DT,DT)]^{\op}.$
\end{itemize}
\end{thm}

We call the pair $(\G,DT)$ the \emph{Koszul dual} of $(\L,T)$.

Since $T$ is graded self-orthogonal, we can construct a certain bigraded bimodule $X$ and a functor between unbounded derived categories of graded modules $G_T=\mathbb R \Hom_{\Gr \L}(X,-) \colon \mathcal D (\Gr \L) \to \mathcal D(\Gr \G)$, as explained in \cite[Section 3]{Adv}. When $\L$ is $T$-Koszul, the functor $G_T$ restricts to an equivalence in different ways and we mention one version here. Let $\mathcal F_{\gr \L}(T)$ be the full subcategory of $\gr \L$ consisting of modules $M$ having a finite filtration $0=M_0 \subseteq M_1 \subseteq \ldots \subseteq M_t=M$ where the factors $M_i/M_{i-1}$ are graded shifts of direct summands of $T$ for all $1 \leq i \leq t$. Let $\mathcal L^b(\G)$ denote the category of bounded linear cochain complexes of graded projective $\G$-modules and cochain maps.

\begin{thm}\cite[Theorem 4.3.2]{Adv}\label{delta}
The functor $G_T \colon \mathcal D (\Gr \L) \to \mathcal D(\Gr \G)$ restricts to an equivalence $G_T \colon \mathcal F_{\gr \L}(T) \rightarrow \mathcal L^b(\G)$.
\end{thm}

With some finiteness conditions, we get an equivalence of bounded derived categories of finitely generated graded modules.

\begin{thm}\cite[Theorem 4.3.4]{Adv}\label{derequiv}
Suppose $\L$ is artinian and $\G$ is noetherian. Assume $\gldim \G < \infty$. Then there is an equivalence of triangulated categories $G_T^b \colon \mathcal D^b(\gr \L) \rightarrow \mathcal D^b(\gr \G)$.
\end{thm}

We can try to fit a standard Koszul algebra into the $T$-Koszul framework by setting $T=\Delta$. For this to work the algebra has to be regraded. It turns out that many standard Koszul algebras after an appropriate regrading are Koszul with respect to $\Delta$. We begin with an example.

\begin{example}\label{cato}
Let $\L$ be the path algebra  $\L=\Bbbk Q/I$, where $Q$ is the quiver
$$\xymatrix{1 \ar@/^/[r]^\alpha & 2 \ar@/^/[l]^{\alpha^\circ} \ar@/^/[r]^\beta & 3 \ar@/^/[l]^{\beta^\circ}}$$
and $I=\langle \rho \rangle$ is the ideal generated by the set of relations $$\rho= \{ \alpha \alpha^\circ - \beta^\circ \beta, \beta \beta^\circ\}.$$ This is a quasi-hereditary algebra with duality. If $\Bbbk=\mathbb C$, the field of complex numbers, then $\L$ corresponds to a non-trivial singular block of the BGG category $\mathcal O$ for the semi-simple Lie algebra ${\mathsf {sl}}(3,\mathbb C)$ \cite[5.2]{Quiv}. (\cite{Hum} is a textbook reference for category $\mathcal O$.) The indecomposable projective modules are
$$\begin{array}{ccccccccccc}
P_1 \colon & \xymatrix@!=2pt{\mathtt{S_1} \ar@{-}[dr] &&\\
& \mathtt{S_2} \ar@{-}[dl] \ar@{-}[dr] & \\ \mathtt{S_1} \ar@{-}[dr] && \mathtt{S_3}, \ar@{-}[dl] \\ & \mathtt{S_2} \ar@{-}[dl] &\\ \mathtt{S_1}} && P_2 \colon & \xymatrix@!=2pt{& \mathtt{S_2} \ar@{-}[dl] \ar@{-}[dr] &
\\\mathtt{S_1} \ar@{-}[dr]&& \mathtt{S_3}, \ar@{-}[dl]\\& \mathtt{S_2} \ar@{-}[dl] &\\ \mathtt{S_1}} && P_3 \colon & \xymatrix@!=2pt{&& \mathtt{S_3}
\ar@{-}[dl]\\& \mathtt{S_2} \ar@{-}[dl] &\\ \mathtt{S_1}}. \end{array}$$
The standard modules are $\Delta_1=S_1$,
$$\begin{array}{ccccccccccc}
\Delta_2 \colon & \xymatrix@!=2pt{& \mathtt{S_2} \ar@{-}[dl]\\
\mathtt{S_1}}, \end{array}$$ and $\Delta_3=P_3$.
The characteristic tilting module $\mathbb T=T_1 \oplus T_2 \oplus T_3$ has direct summands $T_1=S_1$,
$$\begin{array}{ccccccccccc}
T_2 \colon & \xymatrix@!=2pt{\mathtt{S_1} \ar@{-}[dr] &\\ & \mathtt{S_2 } ,
\ar@{-}[dl]\\\mathtt{S_1}}
\end{array}$$ and $T_3=P_3$.

The algebra is $T$-Koszul in three important ways using three different sets of orthogonal modules: namely, (i) the indecomposable summands of the characteristic tilting module, (ii) the simple modules, and (iii) the standard modules. The grading imposed on $\L$ is different in each case.

(i) If all arrows are assigned degree $0$, then $\L=\L_0$ and $\L$ is Koszul with respect to $\mathbb T$. The Koszul dual algebra of $(\L,\mathbb T)$ is the \emph{Ringel dual} quasi-hereditary algebra $[\End_\L(\mathbb T)]^{\op}$. In this example $\L$ is \emph{Ringel self-dual}, that is, there is an isomorphism $[\End_\L(\mathbb T)]^{\op} \cong \L$.

(ii) If all arrows are assigned degree $1$, then $\L_0 \cong \Bbbk^{\times r}$ and $\L$ is a Koszul algebra in the classical sense. It is even a standard Koszul algebra since the standard modules have linear projective resolutions
$$0 \to P_2 \gsh 1 \to P_1 \to \Delta_1 \to 0,$$
$$0 \to P_3 \gsh 1 \to P_2 \to \Delta_2 \to 0,$$
$$0 \to P_3 \to \Delta_3 \to 0.$$
The Koszul dual algebra $[\Ext_\L^{\ast} (\L_0,\L_0)]^{\op}$ is isomorphic to $\Bbbk Q/I'$, where $I'=\langle \rho ' \rangle$ is the ideal generated by the set of relations $$\rho '= \{\beta \alpha, \alpha \alpha^\circ - \beta^\circ \beta, \beta \beta^\circ, \alpha^\circ \beta^\circ \}.$$ If $\Bbbk=\mathbb C$, then the Koszul duality here is an instance of the parabolic-singular duality from \cite{Bei}, and the algebra $\Bbbk Q/I'$ must correspond to a block of a certain parabolic subcategory $\mathcal O^{\mathfrak p}$.

(iii) Finally, we want to show that $\L$ is Koszul with respect to $\Delta$. For this we need to do a little trick with the grading. Let $\deg \alpha = \deg \beta = 1$ and $\deg \alpha^\circ = \deg \beta^\circ = 0$. (See \cite{Flo} for a very similar choice of grading in a parabolic setting.) Then $\L$ is Koszul with respect to $\L_0 \simeq \Delta=\Delta_1 \oplus \Delta_2 \oplus \Delta_3$ since $\Ext^i_{\Gr \L}(\Delta,\Delta \gsh j) \neq 0$ implies $i=j$. The Koszul dual algebra $\G=[\Ext_\L^{\ast} (\Delta,\Delta)]^{\op}$ is isomorphic to the graded algebra $\Bbbk \check Q/\check I$, where $\check Q$ is the quiver $$\xymatrix{1 & 2 \ar@/_/[l]_{\check \alpha} \ar@/^/[l]^{\alpha^\circ} & 3 \ar@/_/[l]_{\check \beta} \ar@/^/[l]^{\beta^\circ}},$$ the grading is given by $\deg \check \alpha = \deg \check \beta = 1$ and $\deg \alpha^\circ = \deg \beta^\circ = 0$, and $\check I=\langle \check \rho \rangle$ is the ideal generated by the set of relations $$\check \rho= \{\check \alpha \check \beta, \alpha^\circ \check \beta - \check \alpha \beta^\circ \}.$$ We observe that $\G$ is a directed algebra. From Theorem \ref{kosdua} we know there is an isomorphism of graded algebras $$\L \cong [\Ext_\G^{\ast} (D \Delta,D\Delta)]^{\op}.$$ According to Theorem \ref{derequiv} there is an equivalence of triangulated categories $$\mathcal D^b(\gr \L) \to \mathcal D^b(\gr \G).$$ It restricts to the equivalence $$\mathcal F_{\gr \L}(\Delta) \to \mathcal L^b(\G)$$ from Theorem \ref{delta}.
\end{example}

\section{Height functions}

With the appropriate grading, the algebra in Example \ref{cato} was shown to be Koszul with respect to $\Delta$. It would be a mistake to assume that all standard Koszul algebras are Koszul with respect to $\Delta$ in a similar fashion. A sufficient condition can be given in terms of \emph{height functions}. A height function is a function $h\colon \{1,\ldots, r\} \to \mathbb N_{\geq 0}$.

Let $\Lambda=\bigoplus_{i=0}^t \L_i$ be a standard Koszul algebra. Fix a height function $h$. We consider the following condition on $h$.
\begin{equation}
[(\Delta_j)_l \colon S_i]=0 \text{ whenever } h(i) \neq h(j)-l. \tag{H}
\end{equation}
Since $\Delta_j$ is generated in a single degree, the graded parts of $\Delta_j$ coincide with the radical layers. Condition (H)
says that the height function contains the information about which radical layer each composition factor of $\Delta_j$ belongs to. In Example \ref{cato}, the height function $h(i)=i$ for all $1 \leq i \leq 3$ satisfies condition (H).

In \cite{Stan}, the standard modules are said to be \emph{directed} if condition (H) is satisfied for some height function $h$.

The existence of a height function satisfying (H) is more important than its actual values. The following lemma explains how values are related within the same block of the algebra.

\begin{lem}\label{h1}
If $h$ satisfies condition {\rm{(H)}}, then $e_i \L_1 e_j \neq 0$ implies $|h(i) - h(j)|=1$. More precisely, suppose {\rm{(H)}} holds and $e_i \L_1 e_j \neq 0$. Then
\begin{itemize}
\item[(a)] $h(i) - h(j)=1$ if and only if $j<i$.
\item[(b)] $h(j) - h(i)=1$ if and only if $i<j$.
\end{itemize}
\end{lem}

\begin{proof}
Suppose (H) holds and $e_i \L_1 e_j \neq 0$. If $i<j$, then $[(\Delta_j)_1 \colon S_i] \neq 0$ and $h(i)=h(j)-1$. If $i>j$, by using the duality $(-)^{\circ}$ we get that $e_j \L_1 e_i \neq 0$, and by repeating the argument we get $h(j)=h(i)-1$. We cannot have $i=j$, since $\End_\L(\Delta_i)$ is a division ring.
\end{proof}

In order to obtain Koszulity with respect to $\Delta$, we follow the same strategy as in Example \ref{cato}, and regrade our algebra $\L$. Given a height function $h$ satisfying (H), we define what we call the \emph{$\Delta$-grading} on $\L$ in the following way. We assign degrees to the spaces $e_i \L_l e_j$, $l \geq 0$, $1 \leq i,j \leq r$ whenever $e_i \L_l e_j \neq 0$. The idempotents should be in degree $0$, so we put $\deg_{\Delta}(e_i \L_0 e_i)=0$ for all $1 \leq i \leq r$. Obviously, we have $e_i \L_0 e_j=0$ whenever $i \neq j$. Since $\L$ is standard Koszul, it is generated by $\L_1$ over $\L_0$. With Lemma \ref{h1} in mind, we define the new grading on generators by
$$\deg_{\Delta}(e_i \L_1 e_j)=\left \{
\begin{array}{c c}
1 & \text { if } h(i)-h(j)=1\\
0 & \text { if } h(j)-h(i)=1
\end{array}
\right .$$
for all $1 \leq i,j \leq r$.
This can most conveniently be expressed as $$\deg_{\Delta}(e_i \L_1 e_j)=\frac {1+h(i)-h(j)} 2.$$
Hence, for all $l \geq 0$, $1 \leq i,j \leq r$ such that $e_i \L_l e_j \neq 0$, we get the formula $$\deg_{\Delta}(e_i \L_l e_j)=\frac {l+h(i)-h(j)} 2 \geq 0.$$ In order to distinguish the $\Delta$-grading from the original grading we use the notation $\L_{[n]}$ for the degree $n$ part of $\L$ with the $\Delta$-grading. The dependence of the $\Delta$-grading on the height function is only apparent, as the next lemma shows.

\begin{lem}
As long as {\rm{(H)}} holds, the $\Delta$-grading does not depend on the particular choice of a height function $h$.
\end{lem}

\begin{proof}
For simplicity assume that $\L$ is indecomposable as an algebra. Let $h$ and $h'$ be two height functions satisfying (H), and let $d=h(1)-h'(1)$. By repeated use of Lemma \ref{h1}, we get that $d=h(i)-h'(i)$ for all $1\leq i \leq r$. Then $\frac {l+h(i)-h(j)} 2= \frac {l+h'(i)-h'(j)} 2$ for all $l \geq 0$, $1 \leq i,j \leq r$. The general case follows.
\end{proof}

The degree zero part of $\L$ under the $\Delta$-grading is described in the following proposition.

\begin{prop}\label{null}
If {\rm{(H)}} holds and $\L$ is given the $\Delta$-grading, then $\L_{[0]} \simeq \Delta$ as graded $\L$-modules.
\end{prop}

\begin{proof}
It is sufficient to prove that $\L_{[0]} e_j \simeq \Delta_j$ for any given $1 \leq j \leq r$. Both $\L_{[0]} e_j$ and $\Delta_j$ are quotients of $P_j=\L e_j$. Let $K_j$ denote the kernel of $P_j \twoheadrightarrow \Delta_j$. Consider an element $e_i \lambda_l e_j \in (\L_{[\geq 1]}) e_j$ and suppose $\lambda_l \in \L_l$. Then $h(i)>h(j)-l$, so $e_i \lambda_l e_j \in K_j$ by (H). So $\Delta_j$ is a quotient of $\L_{[0]} e_j$.

Let $e_i \lambda_l e_j \in e_i \L_l e_j$ with $l \geq 1$. If $l=1$ and $e_i \lambda_1 e_j \neq 0$, then $\deg_{\Delta}(e_i \lambda_1 e_j)=0$ implies $i<j$. Since the generators of $\L$ are in $\L_1$, by induction on $l \geq 1$ it follows that $\deg_{\Delta}(e_i \lambda_l e_j)=0$ implies $i<j$. The module $K_j$ is generated by elements of the form $e_i \lambda_l e_j$ with $\lambda_l \in \L_l$, $l \geq 1$ and $i>j$. Whenever $i>j$, we have $e_i \lambda_l e_j \in (\L_{[\geq 1]}) e_j$. Therefore $K_j=(\L_{[\geq 1]}) e_j$ and $\L_{[0]} e_j \simeq \Delta_j$.
\end{proof}

As a consequence we have established that $\Delta$ is a tilting $\L_{[0]}$-module, which is one of the conditions for $\L$ being Koszul with respect to $\Delta$. Our next task is to show that $\Delta$ is graded self-orthogonal.

\begin{prop}\label{kazh}
Let $\L$ be a standard Koszul algebra that admits a height function $h$ satisfying condition {\rm{(H)}}.
\begin{itemize}
\item[(a)] Consider $\L$ with the ordinary grading. If $\Ext^u_{\Gr \L}(\Delta_j,S_i \gsh v) \neq 0$, then $u=v=h(i)-h(j)$. Similarly, if $\Ext^u_{\Gr \L}(S_i \gsh {-v},\nabla_j) \neq 0$, then $u=v=h(i)-h(j)$.
\item[(b)] When $\L$ is regraded according to the $\Delta$-grading, each standard module $\Delta_j$ has a linear projective resolutions.
\item[(c)] When $\L$ is regraded according to the $\Delta$-grading, each costandard module $\nabla_j$ has an injective coresolution $$0 \to \nabla_j \to I^0 \to I^1 \to I^2 \to \ldots $$ with $I^p$ cogenerated in degree $0$ for all $p \geq 0$.
\end{itemize}
\end{prop}

\begin{proof}
(a) For the first part suppose $\Ext^u_{\Gr \L}(\Delta_j,S_i \gsh v) \neq 0$. Since $\L$ is standard Koszul, we have $u=v$. If $P_a \to P_b$ is a linear map between indecomposable projective $\L$-modules, then $|h(a) - h(b)|=1$ by Lemma \ref{h1}. Since $\Delta_j$ has a linear projective resolution $$\ldots \to P^u \to \ldots \to P^2 \to P^1 \to P^0  \to \Delta_j \to 0$$ with $P^0=P_j$, and the indecomposable projective module $P_i$ occurs as a direct summand of $P^u$, we must have $h(i)-h(j) \leq u$. On the other hand, according to \cite[Lemma 3]{Far} we have $h(i)-h(j) \geq u$. So $h(i)-h(j)=u$. The second part follows by applying the duality $(-)^\circ$.

(b) In the ordinary grading the graded projective resolution of a standard module $\Delta_j$ is linear, so at each step of the resolution the image of generators are linear combinations of elements of the form $e_a \lambda_1 e_b$ with $\lambda_1 \in \L_1$. From (a) it follows that these elements satisfy $h(a)-h(b)=1$. Computing the $\Delta$-degree of these elements we get $\deg_{\Delta}(e_a \L_1 e_b)=\frac {1+h(a)-h(b)} 2= \frac {1+1} 2=1$, so the projective resolution of $\Delta_j$ is also linear when $\L$ is regraded according to the $\Delta$-grading.

(c) From the second part of (a) it follows that the injective resolution of $\nabla_j$ is colinear in the ordinary grading. If $x=e_a x$ is an element in $I^{p-1}$ that maps to a non-zero element in the socle of $I^p$, comparing the heights of the respective socles we find there must be an element of the form $e_b \lambda_1 e_a \in \L_1$ such that $h(a)-h(b)=1$ and $e_b \lambda_1 e_a x$ is a non-zero element in the socle of $I^{p-1}$. Computing the $\Delta$-degree we get $\deg_{\Delta}(e_b \L_1 e_a)=\frac {1+h(b)-h(a)} 2= \frac {1-1} 2=0$. We conclude that for each indecomposable summand of $I^p$ there is an indecomposable summand of $I^{p-1}$ which is cogenerated in the same $\Delta$-degree. Since $I^0=I_j$ is cogenerated in $\Delta$-degree $0$, the statement follows.
\end{proof}

\section{Koszulity with respect to $\Delta$}

We are now ready to prove our main theorem.

\begin{thm}\label{hoved}
Let $\L$ be a standard Koszul algebra. Let $h$ be a height function satisfying condition {\rm{(H)}}. Regrade $\L$ according to the $\Delta$-grading. Then $\L$ is a Koszul algebra with respect to $\L_{[0]} \simeq \Delta$.
\end{thm}

\begin{proof}
We have $\gldim \L_{[0]} \leq \gldim \L < \infty$. The isomorphism $\L_{[0]} \simeq \Delta$ from Proposition \ref{null} shows that $\Delta$ is a tilting $\L_{[0]}$-module. Let $0 \to P^n \to \ldots \to P^2 \to P^1 \to P^0  \to \Delta \to 0$ be a minimal graded projective resolution of $\Delta$. According to Proposition \ref{kazh}(b), the projective module $P^i$ is generated in degree $i$. Since $\Delta \gsh j$ is concentrated in a degree $j$, we have $\Hom_{\Gr \L}(P^i,\Delta \gsh j)=0$ whenever $i \neq j$. It follows that $\Ext^i_{\Gr \L}(\Delta,\Delta \gsh j)=0$ whenever $i \neq j$.
\end{proof}

If $\L$ and $h$ are as in the above theorem, then in particular $\Delta$ is a graded self-orthogonal $\L$-module. As usual, let $\G=[\Ext_{\L}^\ast(\Delta,\Delta)]^{\op}$. We construct the functor $G_{\Delta}=\mathbb R \Hom_{\Gr \L}(X,-) \colon \mathcal D(\Gr \L) \to \mathcal D(\Gr \G)$ in the usual way mentioned after Theorem \ref{kosdua}. The indecomposable graded projective $\G$-modules generated in degree zero are of the form $G_{\Delta}(\Delta_i)$, $1 \leq i \leq r$.

\begin{prop}
Let $\L$ and $h$ be as in Theorem \ref{hoved}. For any $1 \leq i \leq r$,
$$G_{\Delta}(\nabla_i) \simeq {}_\G S_i,$$
where ${}_\G S_i$ is the simple top of the indecomposable graded projective $\G$-module $G_{\Delta}(\Delta_i)$.
\end{prop}

\begin{proof}
According to \cite[Proposition 3.2.1(e)]{Adv}, we have $$(H^l G_{\Delta}(\nabla_i))_j \simeq \Ext^{l+j}_{\Gr \L}(\Delta,\nabla_i \gsh j)=0$$ whenever $j \neq 0$ or $l \neq 0$. So
\begin{align*}
G_{\Delta}(\nabla_i) &\simeq (H^0 G_{\Delta}(\nabla_i))_0\\
&\simeq \Hom_{\Gr \L}(\Delta,\nabla_i)\\
&\simeq \Hom_{\Gr \L}(\Delta_i,\nabla_i),
\end{align*}
which is $1$-dimensional as a $\Bbbk$-vector space. If $i \neq s$, then $$\Hom_{\mathcal D\Gr \G}(G_{\Delta}(\Delta_s),G_{\Delta}(\nabla_i)) \simeq \Hom_{\mathcal D\Gr \L}(\Delta_s,\nabla_i)=0,$$ so we must have $G_{\Delta}(\nabla_i) \simeq \top G_{\Delta}(\Delta_i)$.
\end{proof}

When we know that $\L$ is Koszul with respect to $\Delta$, we can apply all the $T$-Koszul machinery mentioned earlier in this paper and get a surprisingly clean statement of the duality theory. According to Theorem \ref{kosdua}, the $\Delta$-grading and the $\Ext$-grading on $\L$ coincide, so the results can be stated without reference to the $\Delta$-grading.

\begin{cor}\label{gqh}
Let $\L$ be a standard Koszul algebra admitting a height function $h$ satisfying condition {\rm{(H)}}. Let $$\G=[\Ext_{\L}^{\ast} (\Delta,\Delta)]^{\op}.$$ Then $$\L \cong [\Ext_{\G}^{\ast} (D\Delta,D\Delta)]^{\op}$$ as ungraded algebras. Furthermore, when $\L$ and $\G$ are given the $\Ext$-grading, then there is an equivalence of triangulated categories $\mathcal D^b(\gr \L) \to \mathcal D^b(\gr \G)$ which restricts to an equivalence $\mathcal F_{\gr \L}(\Delta) \to \mathcal L^b(\G).$
\end{cor}

\begin{proof}
According to Theorem \ref{hoved}, the algebra $\L$ with the $\Delta$-grading is Koszul with respect to $\Delta$. Since $\L$ has finite global dimension, the algebra $\G=[\Ext_{\L}^{\ast} (\Delta,\Delta)]^{\op}$ is finite dimensional. Since $\G$ is directed, we must have $\gldim \G < \infty$. Since $\Delta \in \mathcal D^b(\gr \L)$, the category $\mathcal F_{\gr \L}(\Delta)$ is a subcategory of $\mathcal D^b(\gr \L)$. The rest follows from Theorems \ref{kosdua}, \ref{derequiv} and \ref{delta}.
\end{proof}

This corollary should be compared with Theorem 4.1 in \cite{Stan}. In our treatment Koszul duality is between two graded algebra and we avoid going to Koszul categories. It should be noted however that we assume the presence of a duality functor $(-)^\circ$, an assumption not made in \cite{Stan}. When $\L$ is a graded quasi-hereditary algebra with duality, then the conditions (I)--(IV) in \cite{Stan} imply that $\L$ is a standard Koszul algebra admitting a height function $h$ satisfying condition {\rm{(H)}} \cite[Propositions 3.6 and 3.7]{Stan}.

If $\L$ is a standard Koszul algebra admitting a height function $h$ satisfying condition {\rm{(H)}}, then Corollary \ref{gqh} combined with Theorem \ref{kosdua}(a) tell us that $\G=[\Ext_{\L}^{\ast} (\Delta,\Delta)]^{\op}$ is Koszul with respect to $_{\G}D\Delta$. A perhaps more interesting question considered in \cite{Maz} and \cite{Stan} is whether the extension algebra of $\Delta$ is Koszul in the classical sense. We prove that with our conditions this is indeed the case.

\begin{thm}\label{ord}
Let $\L$ be a standard Koszul algebra admitting a height function $h$ satisfying condition {\rm{(H)}}. Then $\G=[\Ext_{\L}^{\ast} (\Delta,\Delta)]^{\op}$ is a Koszul algebra in the classical sense.
\end{thm}

\begin{proof}
From condition {\rm{(H)}} it follows that $\Hom_\L(\Delta_i,\Delta_j)\neq 0$ with $i \neq j$ implies $h(i)< h(j)$. If $\Ext^u_\L(\Delta_i,\Delta_j)\neq 0$ for some $u>0$, then $\Hom_\L(P_a,\Delta_j)\neq 0$ for some indecomposable projective direct summand $P_a$ of $P^u$, where $P^u$ is the projective module at position $u$ in the projective resolution of $\Delta_i$. From Proposition \ref{kazh}(a) it follows that $h(i)<h(a)$. Condition {\rm{(H)}} implies $h(i)<h(a) \leq h(j)$. With these inequalities in mind we define a new grading $\deg_H$ on $\G$ by the formula $$\deg_H(\Ext^{\ast}_\L(\Delta_i,\Delta_j))=h(j)-h(i).$$
Let $\G_{\{i\}}$ denote the degree $i$ part of $\G$ with this new grading. It follows from the inequalities above that $\G=\bigoplus_{i \geq 0} \G_{\{i\}}$ and $$\G_{\{0\}}=\bigoplus_{i=1}^r[\End_\L(\Delta_i)]^{\op} \cong \Bbbk^{\times r}.$$ Hence if in the new grading $P_i \gsh n \to P_j$ is a map of indecomposable projective $\G$-modules, then $h(j)-h(i)=n$. The Koszulity condition for $\G$ can now be reformulated as $$\Ext^w_\G(S_i,S_j)=0 \text{ whenever } w \neq h(i)-h(j).$$

Assume $\Ext^w_\G(S_i,S_j) \neq 0$. Then there must exist $m \in \Z$ such that $\Ext^w_{\Gr \G}(S_i,S_j \gsh m) \neq 0$ in the $\Ext$-grading of Corollary \ref{gqh}. Applying the derived equivalence we get
\begin{align*}
0 \neq \Ext^w_{\Gr \G}(S_i,S_j \gsh m) &\simeq \Hom_{\mathcal D\Gr \G}(S_i,S_j \gsh m[w])\\ &\simeq \Hom_{\mathcal D\Gr \L}(\nabla_i,\nabla_j \gsh {-m}[w-m])\\ &\simeq \Ext^{w-m}_{\Gr \L}(\nabla_i \gsh m,\nabla_j)
\end{align*}

Let $v=w-m$. If $\Ext^v_{\Gr \L}(\nabla_i \gsh m,\nabla_j) \neq 0$, then $\Hom_{\Gr \L}(\nabla_i \gsh m,I_b) \neq 0$ for some indecomposable injective direct summand $I_b$ of $I^v$, where $I^v$ is the injective module at position $v$ in the injective coresolution of $\nabla_j$. According to Proposition \ref{kazh}(a)(c), the injective module $I^v$ and therefore $I_b$ are cogenerated in degree $0$ and $v=h(b)-h(j)$.

Choose a non-zero $f \in \Hom_{\Gr \L}(\nabla_i \gsh m,I_b)$, and choose an element $x=e_b x \in \nabla_i \gsh m$ such that $f(x)\neq 0$ is in the socle of $I_b$. There is an element $e_i \lambda_l e_b \in \L_{[m]}$ such that $e_i \lambda_m e_b x \neq 0$ is in the socle of $\nabla_i \gsh m$. Solving for $l$ in the $\Delta$-degree formula we get $e_i \lambda_m e_b \in \L_l$, where $l=2m-h(i)+h(b)$.

Applying the duality $(-)^\circ$ to condition {\rm{(H)}}, we get the dual condition
\begin{equation}
[(\nabla_i)_{-l} \colon S_j]=0 \text{ whenever } h(j) \neq h(i)-l. \tag{H$^\circ$}
\end{equation}
In our situation $(\nabla_i)_{-l}$ maps to the socle of $I_b$, so condition {\rm{(H$^\circ$)}} implies $l=h(i)-h(b)$.

We now have four equations:
\begin{align*}
v&=w-m,\\ v&=h(b)-h(j),\\ l&=2m-h(i)+h(b),\\ l&=h(i)-h(b).
\end{align*}
Combining the last two we get $m=l=h(i)-h(b)$. So $w=v+m=h(b)-h(j)+h(i)-h(b)=h(i)-h(j)$, and we have proven that $\G$ is Koszul.
\end{proof}

We now discuss examples where Theorem \ref{hoved} and Corollary \ref{gqh} can be applied. A quasi-hereditary algebra with duality is called \emph{multiplicity free} if $(P_i \colon \Delta_j)=[\Delta_j \colon S_i] \leq 1$ for all $1 \leq i,j \leq r$.

\begin{example}
In \cite{Stan} the authors considered algebras corresponding to blocks of the BGG category $\mathcal O$. The algebras arising in this way are known to be quasi-hereditary with duality. In this context the standard modules are the same as the so-called Verma modules. The height function can be defined via the Weyl group. If such an algebra is multiplicity free, the conditions (I)--(IV) in \cite{Stan} are satisfied and the algebra is therefore a standard Koszul algebra admitting a height function $h$ satisfying condition {\rm{(H)}}.

We already saw such an algebra in Example \ref{cato}. Another example can be found in \cite[3.1]{Coin}. The following algebra $\L=\mathbb C Q/I$ corresponds to the principal block of category $\mathcal O$ for the semi-simple Lie algebra ${\mathsf {so}}(4,\mathbb C)$. The quiver $Q$ is
$$
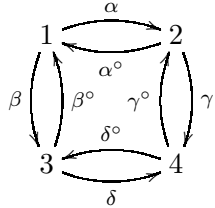
\xymatrix@!=25pt{1 \ar@/^/[r]^\alpha \ar@/_/[d]_\beta & 2 \ar@/^/[l]^{\alpha^\circ} \ar@/^/[d]^\gamma\\ 3 \ar@/_/[r]_\delta \ar@/_/[u]_{\beta^\circ} & 4 \ar@/_/[l]_{\delta^\circ} \ar@/^/[u]^{\gamma^\circ}}$$ and $I=\langle \rho \rangle$ is the ideal generated by the set of relations $$\rho=\{\gamma \alpha - \delta \beta, \alpha \alpha^\circ, \beta \beta^\circ, \gamma \gamma^\circ, \delta \delta^\circ, \alpha \beta^\circ - \gamma^\circ \delta, \beta \alpha^\circ - \delta^\circ \gamma, \alpha^\circ \gamma^\circ - \beta^\circ \delta^\circ\}.$$ The standard modules are $\Delta_1=S_1$, $$\begin{array}{ccccccccccc}
\Delta_2 \colon & \xymatrix@!=2pt{& \mathtt{S_2}, \ar@{-}[dl] \\ \mathtt{S_1}} && \Delta_3 \colon & \xymatrix@!=2pt{\mathtt{S_3} \ar@{-}[dr] &
\\& \mathtt{S_1},} && \Delta_4 \colon & \xymatrix@!=2pt{& \mathtt{S_4} \ar@{-}[dl] \ar@{-}[dr] &
\\\mathtt{S_3} \ar@{-}[dr]&& \mathtt{S_2}. \ar@{-}[dl]\\& \mathtt{S_1}} \end{array}$$ The algebra $\L$ is Koszul relative to $\Delta$. The Koszul dual algebra $\G=[\Ext_\L^{\ast} (\Delta,\Delta)]^{\op}$ is isomorphic to the algebra $\mathbb C \check Q/\check I$, where $\check Q$ is the quiver $$
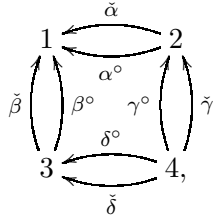
\xymatrix@!=25pt{1 & 2 \ar@/^/[l]^{\alpha^\circ} \ar@/_/[l]_{\check \alpha}\\ 3 \ar@/_/[u]_{\beta^\circ} \ar@/^/[u]^{\check \beta} & 4, \ar@/_/[l]_{\delta^\circ} \ar@/^/[l]^{\check \delta} \ar@/^/[u]^{\gamma^\circ} \ar@/_/[u]_{\check \gamma}}$$ and $\check I=\langle \check \rho \rangle$ is the ideal generated by the set of relations $$\check \rho= \{\check \alpha \check \gamma - \check \beta \check \delta, \alpha^\circ \check \gamma - \check \beta \delta^\circ, \check \alpha \gamma^\circ - \beta^\circ \check \delta, \alpha^\circ \gamma^\circ - \beta^\circ \delta^\circ \}.$$

As usual there is an isomorphism $\L \cong [\Ext_{\G}^{\ast} (D\Delta,D\Delta)]^{\op}$. In this and other examples we have labeled the arrows such that when $\L$ and $\G$ are given the $\Ext$-grading, the arrows marked with ${}^\circ$ have degree $0$ while all other arrows have degree $1$. With this grading there is an equivalence $\mathcal D^b(\gr \L) \to \mathcal D^b(\gr \G)$ which restricts to an equivalence $\mathcal F_{\gr \L}(\Delta) \to \mathcal L^b(\G).$
\end{example}

\begin{example}
Starting with a multiplicity free algebra corresponding to a singular block of category $\mathcal O$, we can apply parabolic-singular duality and get a new algebra which according to \cite[Theorem 5.1]{Stan} is also multiplicity free and satisfies conditions (I)--(IV). It follows that the new algebra is also Koszul with respect to $\Delta$.

One such algebra is $\L'=\Bbbk Q/I'$ from Example \ref{cato}(ii). The standard modules for $\Bbbk Q/I'$ are $\Delta_1=S_1$,
$$\begin{array}{cccccc}
\Delta_2 \colon & \xymatrix@!=2pt{& \mathtt{S_2} \ar@{-}[dl]\\ \mathtt{S_1} &}, &&
\Delta_3 \colon  & \xymatrix@!=2pt{& \mathtt{S_3}  \ar@{-}[dl]\\ \mathtt{S_2} &} .
\end{array}$$
The Koszul dual algebra $\G'=[\Ext_{\L'}^{\ast} (\Delta,\Delta)]^{\op}$ is isomorphic to the algebra $\Bbbk \check Q/\check I$, where $\check Q$ is the quiver $$
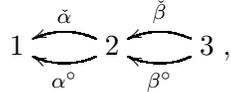
\xymatrix{1 & 2 \ar@/_/[l]_{\check \alpha} \ar@/^/[l]^{\alpha^\circ} & 3 \ar@/_/[l]_{\check \beta} \ar@/^/[l]^{\beta^\circ}},$$ and $\check I=\langle \check \rho \rangle$ is the ideal generated by the set of relations $$\check \rho= \{\alpha^\circ \check \beta - \check \alpha \beta^\circ, \alpha^\circ \beta^\circ \}.$$
\end{example}

If the radical layers of the standard modules are known, the question whether the algebra admits a height function satisfying condition (H) can usually be determined by inspection.

\begin{example}
In \cite{Dipl} and \cite{Par} extension algebras of standard modules are computed inside another parabolic category $\mathcal O^\mathfrak p$, namely the one arising from $\mathfrak g=\mathsf {gl}(m+n,\mathbb C)$ and $\mathfrak p$ the parabolic subalgebra with Levi component $\mathfrak l = \mathsf {gl}(m,\mathbb C) \oplus \mathsf {gl}(n,\mathbb C)$.

If $n=1$, then \cite[Theorem 5.1]{Dipl} the algebra corresponding to the principal block of $\mathcal O^\mathfrak p$ is isomorphic to the algebra $\mathcal A_{m+1}=\mathbb CQ/I$, where $Q$ is the quiver $$\xymatrix{1 \ar@/^/[r]^{\alpha_1} & 2 \ar@/^/[l]^{\alpha_1^\circ} \ar@/^/[r]^{\alpha_2} & 3 \ar@/^/[l]^{\alpha_2^\circ} \ar@/^/[r]^{\alpha_3} & \cdots \ar@/^/[l]^{\alpha_3^\circ} \ar@/^/[r]^{\alpha_m}& m+1 \ar@/^/[l]^{\alpha_m^\circ}}$$ and $I=\langle \rho \rangle$ is the ideal generated by the set of relations $$\rho= \{\alpha_m \alpha_m^\circ \} \cup \{\alpha_i \alpha_{i-1}, \alpha_{i-1} \alpha_{i-1}^{\circ} - \alpha_i^{\circ} \alpha_i, \alpha_{i-1}^{\circ} \alpha_i^{\circ} \}_{1\leq i\leq m}.$$
The height function $h(i)=i$ for all $1 \leq i \leq m+1$ satisfies condition (H), so $\mathcal A_{m+1}$ is Koszul with respect to $\Delta$. The $\Delta$-grading on $\mathcal A_{m+1}$ already made an appearance in the paper \cite{Flo}.

If $n=2$, then the algebra corresponding to the principal block of $\mathcal O^\mathfrak p$ is still multiplicity free standard Koszul, but from its quiver \cite[Theorem 5.11]{Dipl} we can see that it does not admit a height function satisfying condition (H). In more detail, assuming condition (H) the radical layers of standard modules \cite[Table 5.8]{Dipl} tell us that
\begin{align*}
h(m-1|m-3) &= h(m|m-2)+2\\
&= h(m|m-1)+3\\
&= h(m-1|m-3)+2,
\end{align*}
a contradiction. If an algebra $\L$ is Koszul with respect to $\Delta$, then the extension algebra $\G=[\Ext_\L^{\ast} (\Delta,\Delta)]^{\op}$ is formal since $\Delta$ is graded self-orthogonal. In \cite{Par} it is conjectured that the extension algebra in the $n=2$ case is not formal.
\end{example}

\begin{example}
Another source of quasi-hereditary algebras with duality are the blocks of the Schur algebra $S(n,r)$. In \cite[3.1]{Rar} some of the (Morita equivalence classes of) blocks of $S(2,r)$ are represented as quivers with relations. The composition structure of standard modules can also be found in that paper. Coincidentally, the algebras $\mathcal A_k$ from the previous example occur in this setting, the only difference being that now the characteristic of the field is positive. In fact the notation $\mathcal A_k$ stems from the Schur algebra literature \cite{Schur}.

The algebras $\mathcal A_k$ and the graded derived equivalences with extension algebras of $\Delta$ also appear in the paper \cite{Weyl}.

The other blocks of $S(2,r)$ are not Koszul with respect to $\Delta$. Leaving aside the question whether they are standard Koszul, we can easily prove they do not admit a height function satisfying condition (H). Assuming condition (H), with the notation from  \cite{Rar} we get from the radical layers of the standard module $\Delta_5$ that $$h(5)=h(4)+1.$$ From the radical layers of the standard module $\Delta_6$ we get $$h(4)=h(5)+1,$$ a contradiction.
\end{example}

The following example shows that (H) can hold even if $\L$ is not multiplicity free.

\begin{example}
Let $\L$ be the path algebra  $\L=\Bbbk Q/I$, where $Q$ is the quiver
$$
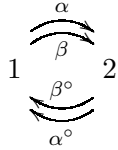
\xymatrix{1 \ar@/^0.8pc/@<1ex>[r]_\beta \ar@/^0.8pc/@<2ex>[r]^\alpha&
2  \ar@/^0.8pc/@<1ex>[l]_{\beta^\circ} \ar@/^0.8pc/@<2ex>[l]^{\alpha^\circ}}$$
and $I=\langle \rho \rangle$ is the ideal generated by the set of relations $\rho=\{\alpha \alpha^\circ, \beta \alpha^\circ,\alpha \beta^\circ, \beta \beta^\circ \}$. This algebra is the \emph{dual extension} (see \cite{DExt} for the definition) of the path algebra of the Kronecker quiver
$$
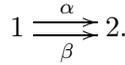
\xymatrix{1 \ar@<-0.5ex>[r]_\beta \ar@<0.5ex>[r]^\alpha&
2.}$$ Algebras that are dual extensions of path algebras with oriented quivers are quasi-hereditary \cite{DExt}.

The indecomposable projective modules over $\L$ are
$$\begin{array}{ccccccccccc}
P_1 \colon & 
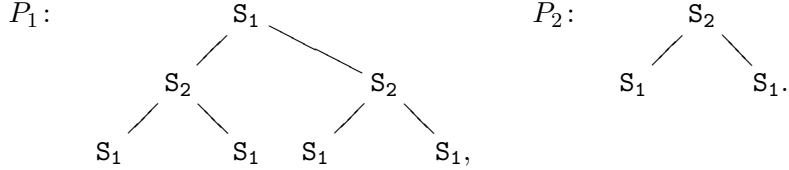
\xymatrix@!=2pt{&& \mathtt{S_1} \ar@{-}[dl] \ar@{-}[drr] &&\\
& \mathtt{S_2} \ar@{-}[dl] \ar@{-}[dr] &&& \mathtt{S_2} \ar@{-}[dl] \ar@{-}[dr]\\ \mathtt{S_1} && \mathtt{S_1} & \mathtt{S_1} && \mathtt{S_1},} && P_2 \colon & \xymatrix@!=2pt{& \mathtt{S_2} \ar@{-}[dl] \ar@{-}[dr] &
\\\mathtt{S_1} && \mathtt{S_1}.}
\end{array}$$
The standard modules are
$$\begin{array}{ccccc}\Delta_1 \colon & 
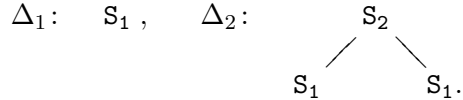
\xymatrix{\mathtt{S_1}}, &&
\Delta_2 \colon  & \xymatrix@!=2pt{& \mathtt{S_2} \ar@{-}[dl] \ar@{-}[dr] &\\\mathtt{S_1} && \mathtt{S_1}.}\end{array}$$
The standard modules have graded projective resolutions $$0 \to P_2\gsh 1 \oplus P_2\gsh 1 \to P_1 \to \Delta_1 \to 0,$$ $$0 \to P_2 \to \Delta_2 \to 0,$$ so $\L$ is standard Koszul. Define a height function by $h(1)=1$ and $h(2)=2$. Then $h$ satisfies (H), and $\L$ is Koszul with respect to $\Delta$.
\end{example}

\end{document}